\pdfoutput=1
\RequirePackage{ifpdf}
\ifpdf % We are running pdfTeX in pdf mode
\documentclass[pdftex]{sigma}
\else
\documentclass{sigma}
\fi

\numberwithin{equation}{section}

\newtheorem{Theorem}{Theorem}[section]
\newtheorem{Proposition}[Theorem]{Proposition}

\newcommand{\Zint}{\mathbb{Z}}
\newcommand{\C}{\mathbb{C}}

\begin{document}
%\allowdisplaybreaks

\newcommand{\arXivNumber}{1712.09564}

\renewcommand{\thefootnote}{}

\renewcommand{\PaperNumber}{061}

\FirstPageHeading

\ShortArticleName{On $q$-Deformations of the Heun Equation}

\ArticleName{On $\boldsymbol{q}$-Deformations of the Heun Equation\footnote{This paper is a~contribution to the Special Issue on Elliptic Hypergeometric Functions and Their Applications. The full collection is available at \href{https://www.emis.de/journals/SIGMA/EHF2017.html}{https://www.emis.de/journals/SIGMA/EHF2017.html}}}

\Author{Kouichi TAKEMURA}

\AuthorNameForHeading{K.~Takemura}

\Address{Department of Mathematics, Faculty of Science and Engineering, Chuo University,\\ 1-13-27 Kasuga, Bunkyo-ku Tokyo 112-8551, Japan}
\Email{\href{mailto:takemura@math.chuo-u.ac.jp}{takemura@math.chuo-u.ac.jp}}

\ArticleDates{Received January 18, 2018, in final form May 29, 2018; Published online June 18, 2018}

\Abstract{The $q$-Heun equation and its variants arise as degenerations of Ruijsenaars--van Diejen operators with one particle. We investigate local properties of these equations. In particular we characterize the variants of the $q$-Heun equation by using analysis of regular singularities. We also consider the quasi-exact solvability of the $q$-Heun equation and its variants. Namely we investigate finite-dimensional subspaces which are invariant under the action of the $q$-Heun operator or variants of the $q$-Heun operator.}


\Classification{39A13; 33E10}

\renewcommand{\thefootnote}{\arabic{footnote}}
\setcounter{footnote}{0}

\section{Introduction}

Accessory parameters of an ordinary linear differential equation are the parameters which are not governed by local data (or local exponents) of the differential equation, and absence or existence of accessory parameters affects the strategy of analyzing the differential equation. A~typical example which does not have an accessory parameter is the hypergeometric equation of Gauss. On the other hand, the Heun equation is an example which has an accessory parameter. Heun's differential equation is a standard form of the second order linear differential equation with four regular singularities on the Riemann sphere, and it is written as
\begin{gather}
\frac{{\rm d}^2y}{{\rm d}z^2} + \left( \frac{\gamma}{z}+\frac{\delta }{z-1}+\frac{\epsilon}{z-t}\right) \frac{{\rm d}y}{{\rm d}z} + \frac{\alpha \beta z -B}{z(z - 1)(z - t)} y= 0,\label{eq:Heun}
\end{gather}
with the condition $\gamma +\delta +\epsilon = \alpha +\beta +1$. The parameter $B$ is an accessory parameter, which is independent of the local exponents.

In \cite{TakR}, $q$-difference deformations of Heun's differential equation were obtained by degene\-ra\-tions of Ruijsenaars--van Diejen operators \cite{vD0,RuiN}, which were also obtained in connection with $q$-Painlev\'e equations \cite{JS,Y}. One of the $q$-deformed difference equations is written as
\begin{gather}
\big(x-q^{h_1 +1/2} t_1\big) \big(x- q^{h_2 +1/2} t_2\big) g(x/q) + q^{\alpha _1 +\alpha _2} \big(x - q^{l_1-1/2}t_1 \big) \big(x - q^{l_2 -1/2} t_2\big) g(qx) \nonumber\\
 \qquad {}-\big\{ \big(q^{\alpha _1} +q^{\alpha _2} \big) x^2 + E x + q^{(h_1 +h_2 + l_1 + l_2 +\alpha _1 +\alpha _2 )/2 } \big( q^{\beta /2}+ q^{-\beta/2}\big) t_1 t_2 \big\} g(x) =0.\label{eq:RuijD5}
\end{gather}
It is characterized as the linear homogeneous $q$-difference equation
\begin{gather}
p^{\langle 0 \rangle} (x) g(x/q) + p^{\langle 1 \rangle}(x) g(x) + p^{\langle 2 \rangle} (x) g(qx)=0 ,
\label{eq:pogp1gp2g}
\end{gather}
where the polynomials $p^{\langle i \rangle} (x) $, $i=0,1,2$, are quadratic $p^{\langle i \rangle} (x) = p^{\langle i \rangle}_0 + p^{\langle i \rangle}_1 x +p^{\langle i \rangle}_2 x^2$ with the condition $p^{\langle 0 \rangle}_0 \neq 0$, $p^{\langle 0 \rangle}_2 \neq 0$, $p^{\langle 2 \rangle}_0 \neq 0$ and $p^{\langle 2 \rangle}_2 \neq 0$. It was discovered by Hahn~\cite{Hahn} in~1971, and we call equation~(\ref{eq:pogp1gp2g}) (or equation~(\ref{eq:RuijD5})) the $q$-Heun equation. Note that Hahn recognized that the parameter $p^{\langle 1 \rangle}_1$ in equation~(\ref{eq:pogp1gp2g}) (or the parameter~$E$ in equation~(\ref{eq:RuijD5})) can be regarded as an accessory parameter.

The $q$-difference equation such that the polynomials $p^{\langle i \rangle} (x) $, $i=0,1,2$, are linear $p^{\langle i \rangle} (x) = p^{\langle i \rangle}_0 + p^{\langle i \rangle}_1 x $ with the condition $p^{\langle 0 \rangle}_0 \neq 0$, $p^{\langle 0 \rangle}_1 \neq 0$, $p^{\langle 2 \rangle}_0 \neq 0$ and $p^{\langle 2 \rangle}_1 \neq 0$ is better known as the $q$-difference hypergeometric equation ($q$-hypergeometric equation), and a standard form is given as
\begin{gather}
(x-q) f(x/q) - ((a+b)x -q-c)f(x)+ (abx-c)f(q x)=0. \label{eq:qHG}
\end{gather}
Then the basic hypergeometric series ($q$-hypergeometric series)
\begin{gather*}
{}_2 \phi _1 (a ,b ;c ;x) = \sum_{n=0}^{\infty} \frac{(a ;q)_n (b ;q)_n }{(q;q) _n (c ;q)_n } x^n, \qquad (\lambda ,q)_n= \prod_{i=0}^{n-1}\big(1- \lambda q^i\big)
\end{gather*}
is a solution to equation~(\ref{eq:qHG}). Note that the $q$-hypergeometric equation appears as a special case of the $q$-Heun equation. Namely, if the parameters in equation~(\ref{eq:RuijD5}) satisfy $l_2 =h_2 +1$ and $E= - \big(q^{\alpha _1} +q^{\alpha _2} \big) q^{h_2 +1/2} t_2 - q^{(h_1 -h_2 + l_1 + l_2 +\alpha _1 +\alpha _2 -1)/2 } \big( q^{\beta /2}+ q^{-\beta/2}\big) t_1$, then the $q$-difference equation has the common factor $\big(x-q^{h_2 +1/2} t_2\big) $ and we obtain the $q$-hypergeometric difference equation.

Recall that the $q$-Heun equation was introduced in~\cite{TakR} as an eigenfunction of the fourth degeneration of the Ruijsenaars--van Diejen operator of one variable
\begin{gather}
 A^{\langle 4 \rangle} = x^{-1} \big(x-q^{h_1 + 1/2} t_1\big) \big(x-q^{h_2 +1/2} t_2\big) T_{q^{-1}}\nonumber\\
\hphantom{A^{\langle 4 \rangle} =}{}
 + q^{\alpha _1 +\alpha _2} x^{-1} \big(x - q^{l_1 -1/2} t_1\big) \big(x - q^{l_2 -1/2} t_2\big) T_q \nonumber \\
\hphantom{A^{\langle 4 \rangle} =}{} -\big\{ \big(q^{\alpha _1} +q^{\alpha _2} \big) x + q^{(h_1 +h_2 + l_1 +l_2 +\alpha _1 +\alpha _2 )/2} \big( q^{\beta/2} + q^{-\beta/2} \big) t_1 t_2 x^{-1} \big\} \label{eq:qH}
\end{gather}
with eigenvalue $E$, where $T_{q^{-1}} g(x) =g(x/q)$ and $T_q g(x)=g(qx) $. Namely equation~(\ref{eq:RuijD5}) is written as
\begin{gather*}
A^{\langle 4 \rangle} g(x) =Eg(x) . %\label{eq:A4E0}
\end{gather*}
The derivation of the $q$-Heun equation in \cite{TakR} was motivated by the relationship between the elliptic $4$-parameter Heun differential operator and the 8-parameter elliptic difference operator (Ruijsenaars--van Diejen operator). The (non-degenerate) Ruijsenaars--van Diejen operator \cite{vD0,RuiN} of one variable is given by
\begin{gather*}
A ^{\langle 0 \rangle} f(z) = V (z) f(z -ia_{-} )+V (-z) f( z+ ia_{-} ) +U (z) f(z),%\label{eq:defeRvD}
\end{gather*}
where
\begin{gather*}
 V (z) = \frac{\prod\limits_{n=1}^8 R_{+}(z-h_n-ia_{-}/2)}{R_{+}(2z+ia_{+}/2)R_{+}(2z-ia_{-}+ia_{+}/2)} ,
\end{gather*}
$R_{+ } (z)= \prod\limits_{k=1}^{\infty}\big(1- e^{-(2k-1) \pi a_{+}} e^{2\pi i z}\big) \big(1-e^{-(2k-1) \pi a_{+}} e^{-2\pi i z} \big)$ is a modified version of the theta function, and $U (z)$ is an elliptic function whose definition is omitted here (see~\cite{TakR}). The operator~$A ^{\langle 0 \rangle} $ has 8 parameters $h_1, \dots ,h_8 $ apart from the parameter $a_+$ in the elliptic function and the difference interval $i a_-$. The multi-variable Ruijsenaars--van Diejen system can be regarded as a relativistic deformation of the quantum system called the Inozemtsev system \cite{vD0,RuiN}. By the non-relativistic limit $a_- \to 0$ of the one-variable Ruijsenaars--van Diejen operator, we obtain the Heun operator in term of the elliptic function written as
\begin{gather*}
 H= -\frac{{\rm d}^2}{{\rm d}x^2} + \sum_{i=0}^3 l_i(l_i+1)\wp (x+\omega_i) ,%\label{eq:Heunellip}
\end{gather*}
where $\wp (x)$ is the Weierstrass elliptic function with basic periods $(2\omega _1 ,2\omega _3)$ and $\{ \omega _0 (=0), \omega _1$, $\omega _2 (=-\omega _1-\omega _3 ), \omega _3 \}$ is a set of half periods (see~\cite{RuiN} for details). Here the 8 parameters in the relativistic system are reduced to the 4 parameters $l_0$, $l_1$, $l_2$, $l_3$ in the non-relativistic limit. It is known that the Heun equation given in equation~(\ref{eq:Heun}) is equivalent to the equation \smash{$H f(x)= E f(x)$}, $E \in \C$ (see~\cite{TakS} for details). Hence the Ruijsenaars--van Diejen system is a~difference analogue of the elliptic form of the Heun equation. A motivation of the paper~\cite{TakR} was a~construction of the difference analogue of the Heun equation without the elliptic expression, and this was realized by considering degenerations four times. After writing the first version of the manuscript of~\cite{TakR}, the author was informed of Hahn's paper~\cite{Hahn} by Professor Ohyama.

In \cite{TakR}, other degenerate operators were also obtained. The third and the second degenerate Ruijsenaars--van Diejen operators are written as
\begin{gather}
 A^{\langle 3 \rangle} = x^{-1} \prod_{n=1}^3 \big(x- q^{h_n+1/2} t_n\big) T_{q^{-1}} + x^{-1}
 \prod_{n=1}^3 (x- q^{l_n -1/2} t_n ) T_q -\big(q^{1/2} +q^{-1/2} \big) x^2\nonumber\\
 \hphantom{A^{\langle 3 \rangle} =}{} +\sum _{n=1}^3 \big( q^{h_n} + q^{l_n} \big)t_n x + q^{(l_1 +l_2 +l_3 +h_1 +h_2 +h_3)/2} \big( q^{\beta/2} + q^{-\beta/2} \big) t_1 t_2 t_3 x^{-1} , \label{eq:qthird} \\
 A^{\langle 2 \rangle} = x^{-2} \prod_{n=1}^4 \big(x- q^{h_n +1/2} t_n\big) T_{q^{-1}}+ x^{-2} \prod_{n=1}^4 \big(x- q^{l_n -1/2} t_n\big) T_{q} \nonumber\\
 \hphantom{A^{\langle 2 \rangle} =}{} -\big(q^{1/2} +q^{-1/2} \big) x^2 + \sum _{n=1}^4 \big(q^{h_n}+ q^{l_n} \big) t_n x\nonumber\\
\hphantom{A^{\langle 2 \rangle} =}{}
+ \prod_{n=1}^4 q^{(h_n +l_n)/2} t_n \left[ - \big(q^{1/2} +q^{-1/2} \big) x^{-2} + \sum _{n=1}^4 \left( \frac{1}{q^{h_n}t_n} + \frac{1}{q^{l_n}t_n} \right) x^{-1} \right] . \label{eq:qsecond}
\end{gather}
The variants of the $q$-Heun equation were introduced as $A^{\langle 3 \rangle} g(x) =Eg(x) $ and $A^{\langle 2 \rangle} g(x) =Eg(x) $, $E \in \C$.

In this paper, we investigate some properties of the $q$-Heun equation and its variants. One of the results is a characterization of variants of the $q$-Heun equation. Recall that the equation $A^{\langle 4 \rangle} g(x) =Eg(x) $ (i.e., equation~(\ref{eq:RuijD5})) is written as equation~(\ref{eq:pogp1gp2g}) whose coefficients are generic quadratic polynomials.

The equation $A^{\langle 3 \rangle} g(x) =Eg(x) $ (resp.\ $A^{\langle 2 \rangle} g(x) =Eg(x) $) is written as
\begin{gather}
a(x) g(x/q) + b(x) g(x) + c(x) g(qx) =0,\label{eq:axgbxgcxg}
\end{gather}
where $a(x)$, $b(x)$, $c(x)$ are polynomials such that $\deg_x a(x)= \deg_x c(x)=3 $, $a(0) \neq 0 \neq c(0)$ and $\deg _x b(x) \leq 3$ (resp.\ $\deg_x a(x)= \deg_x c(x)=4 $, $a(0) \neq 0 \neq c(0)$ and $\deg _x b(x) \leq 4$), although there are relations among the coefficients of the polynomials $a(x)$, $b(x)$, $c(x)$. In this paper we describe these relations using a theory of the regular singularity of the $q$-difference equation. The definition of the regular singularity is analogous to the differential equation (see Section~\ref{sec:regsing}). Let us consider equation~(\ref{eq:axgbxgcxg}) such that $a(x)$, $b(x)$ and $c(x)$ are polynomials, $\deg_x a(x)= \deg_x c(x)=3 $, $\deg _x b(x) \leq 3$ and $a(0)\neq 0 \neq c(0)$. Without loss of generality, we assume that the polynomials $a(x)$ and $c(x)$ are monic (see the argument around equation~(\ref{eq:axgbxgcxh0})). We factorize the polynomials as $a(x)= \big(x- q^{h_1+1/2} t_1\big) \big(x- q^{h_2+1/2} t_2 \big) \big(x- q^{h_3+1/2} t_3\big) $ and $c(x)= \big(x- q^{l_1 -1/2}t_1 \big)\big(x- q^{l_2 -1/2}t_2 \big)\big(x- q^{l_3 -1/2}t_3 \big)$. Write the polynomial $b(x) $ as
\begin{gather*}
 b(x) = b_3 x^3 +b_2 x^2 + b_1 x + b_0 .
\end{gather*}
In the equation $A^{\langle 3 \rangle} g(x) =Eg(x) $ in the form of equation~(\ref{eq:axgbxgcxg}), the constants $b_3$, $b_2$, $b_1$ and $b_0$ are specified. Here we are going to find a characterization of the values $b_3$, $b_2$ and $b_0$ by imposing conditions of local behavior of equation~(\ref{eq:axgbxgcxg}). Namely, if the exponents of equation~(\ref{eq:axgbxgcxg}) about $x=\infty $ are $-1/2$, $1/2$ and the singularity $x=\infty $ is apparent (non-logarithmic), then the values~$b_3$ and~$b_2$ are determined as
\begin{gather*}
 b_3= -\big(q^{1/2} +q^{-1/2} \big), \qquad b_2= \sum _{n=1}^3 \big( q^{h_n} + q^{l_n} \big) t_n. %\label{eq:b320A3Eintro}
\end{gather*}
The value $b_0$ is determined as $b_0 = q^{(l_1 +l_2 +l_3 +h_1 +h_2 +h_3)/2} \big( q^{\beta/2} + q^{-\beta/2} \big) t_1 t_2 t_3 $ by imposing a~condition of the local exponents about $x=0$. See Theorem~\ref{thm:A3} for the precise description. Note that the value $b_1 = E$ is independent of the local conditions about $x=0, \infty$, and it is reasonable to regard~$E$ as an accessory parameter. An analogous result for the equation $A^{\langle 2 \rangle} g(x) =Eg(x) $ is described in Theorem~\ref{thm:A2}.

By a suitable limit in the equation $A^{\langle 3 \rangle} g(x) =Eg(x) $, we obtain a differential equation which is equivalent to Heun's differential equation by a linear fractional transformation (see Section~\ref{eq:locqH}). Then the singularity $x= \infty $ of the equation $A^{\langle 3 \rangle} g(x) =Eg(x) $ corresponds to a regular (non-singular) point of the differential equation in the limit, and we may say that the equation $( A^{\langle 3 \rangle} -E) g(x)=0$ is a $q$-analogue of the Fuchsian equation with four regular singularities $\{ 0, t_1, t_2 ,t_3 \}$, $t_1, t_2, t_3 \in \C _{\neq 0}$. For details see Section~\ref{sec:A3}. Similarly, the equation $\big( A^{\langle 2 \rangle} -E\big) g(x)=0$ is a $q$-analogue of the Fuchsian equation with four regular singularities $\{ t_1, t_2 ,t_3 , t_4 \}$, $t_1, t_2, t_3 , t_4 \in \C _{\neq 0}$ (see Section~\ref{sec:A2}).

In this paper we also investigate special solutions to the $q$-Heun equation and its variants. For this purpose, we consider quasi-exact solvability \cite{Tur} of the equations. Namely, we investigate finite-dimensional invariant spaces which are invariant under the action of the degenerate Ruijsenaars--van Diejen operator $A^{\langle 4 \rangle} $, $A^{\langle 3 \rangle} $ or $A^{\langle 2 \rangle} $. Recall that quasi-exact solvability was applied to investigate the Inozemtsev system of type $BC_N$ \cite{FGGRZ2,TakQ}, which is a generalization of the Heun equation to multiple variables. For the Ruijsenaars--van Diejen system (the Ruijsenaars system of type~$BC_N$), Komori~\cite{Kom} obtained finite-dimensional subspaces spanned by theta functions which are invariant under the action of the Ruijsenaars--van Diejen system. The author believes that our invariant subspaces of the operator $A^{\langle 4 \rangle} $, $A^{\langle 3 \rangle} $ or $A^{\langle 2 \rangle} $ are degenerations of those of Komori, although they are obtained independently and straightforwardly. We hope to develop further our understandings of these invariant subspaces and to find a path for analysis of the (non-degenerate) Ruijsenaars--van Diejen system in the near future.

Here we present a simple example of an invariant subspace for the operator $A^{\langle 4 \rangle}$ in equation~(\ref{eq:qH}). Set $\lambda _1 = (h_1 +h_2 -l_1-l_2- \alpha _1- \alpha _2 -\beta )/2+1 $ and assume that $\lambda _1=-\alpha _1$. Then the operator $A^{\langle 4 \rangle} $ preserves the one-dimensional space spanned by the function $x^{\lambda _1} $, and we have
\begin{gather}
 A^{\langle 4 \rangle} x^{\lambda _1} = -\big\{ \big( q^{h_1 +1/2} t_1 +q^{h_2 +1/2} t_2 \big) q^{ \alpha _1} + \big( q^{l_1-1/2} t_1 +q^{l_2-1/2} t_2 \big) q^{\alpha _2 } \big\} x^{\lambda _1}. \label{eq:A4xla1}
\end{gather}
Therefore the function $x^{\lambda _1 } $ is an eigenfunction of the operator $A^{\langle 4 \rangle} $ with eigenvalue $-\big\{ \big( q^{h_1 +1/2} t_1$ $+q^{h_2 +1/2} t_2 \big) q^{ \alpha _1} + \big( q^{l_1-1/2} t_1 +q^{l_2-1/2} t_2 \big) q^{\alpha _2 } \big\}$ under the assumption $h_1 +h_2 -l_1-l_2+\alpha _1- \alpha _2$ $-\beta +2 = 0 $.

The reminder of the paper is organized as follows. In Section~\ref{sec:regsing}, we explain local properties of solutions to second-order $q$-difference equation about $x=0$ and $x=\infty $. In Section~\ref{eq:locqH}, we investigate local behaviors of solutions about $x=0$ and $x=\infty $ to the $q$-Heun equation and its variants.
We obtain theorems for characterization of $q$-difference equations to variants of the $q$-Heun equation. We also investigate the limit $q\to 1$ and observe relationships with Heun's differential equation. In Section~\ref{sec:inv}, we investigate invariant subspaces related with the $q$-Heun equation and its variants.

In this paper, we assume that $q$ is a positive real number which is not equal to $1$, and all exponents are real numbers.

\section{Regular singularity to the difference equation} \label{sec:regsing}
We investigate the linear difference equation
\begin{gather}
u(x) g(x/q) + v(x) g(x) + w(x) g(qx) =0 \label{eq:axgbxgcxg1}
\end{gather}
such that $u(x)$, $v(x)$, $w(x)$ are Laurent polynomials which are written as
\begin{gather}
 u(x) = \sum_{k=M}^{N} u_k x^k , \qquad v(x) = \sum_{k=M'}^{N'} v_k x^k , \qquad w(x) = \sum_{k=M''}^{N''} w_k x^k,\label{eq:uvw}
\end{gather}
where $ u_{M}$, $v_{M'}$, $w_{M''}$, $u_{N'}$, $v_{N'}$, $w_{N''}$ are all non-zero.

Then the point $x=0$ (resp.\ $x=\infty $) of equation~(\ref{eq:axgbxgcxg1}) is a regular singularity, iff $M =M'' \leq M'$ (resp.\ $N =N'' \geq N' $). The power series solutions about the regular singularity is by now well-understood (see \cite{Adams28} and related papers).

Assume that the point $x=0$ of equation~(\ref{eq:axgbxgcxg1}) is a regular singularity. Then $M=M'' \leq M'$. We investigate local solutions about $x=0$ of the form
\begin{gather*}
 g(x)= x^{\lambda } \sum _{\ell =0}^{\infty} c_{\ell } x^{\ell }, \qquad c_0 \neq 0.
\end{gather*}
If we substitute this into equation~(\ref{eq:axgbxgcxg1}) and set $k+\ell = n +M$, then we get
\begin{gather*}
\sum _{n=0}^{\infty} \sum_{\ell =0}^{n} \big( q^{-\lambda -\ell } u_{n+M-\ell } + v_{n+M-\ell } + q^{\lambda +\ell } w_{n+M-\ell } \big) c_{\ell } x^{n+M +\lambda} =0.
\end{gather*}
Hence
\begin{gather*}
 \sum_{\ell =0}^{n} \big( q^{-\lambda -\ell } u_{n+M-\ell } + v_{n+M-\ell } + q^{\lambda +\ell } w_{n+M-\ell } \big) c_{\ell } =0,
\end{gather*}
where $n$ is a non-negative integer. If $n=0$, then it follows from $c_0 \neq 0$ that
\begin{gather*}
 q^{-\lambda } u_M + v_M + q^{\lambda } w_M =0, \label{eq:expo}
\end{gather*}
which we call the characteristic equation about the regular singularity $x=0$. The exponents about the regular singularity $x=0$ are the values $\lambda $ which satisfy the characteristic equation.
If the quadratic equation $u_M +v_M t + w_M t^2 =0$ has two positive roots, then we have two real-valued exponents. Let~$\lambda '$ be an exponent. If $\lambda ' +n$ ($n \in {\mathbb Z} _{\geq 1} $) is not an exponent, then the coefficient $c_n$ is determined by
\begin{gather*}
\big( q^{-\lambda ' -n} u_M + v_M + q^{\lambda '+n} w_M \big) c_{n} = - \sum_{\ell =0}^{n-1} \big( q^{-\lambda '-\ell } u_{n+M-\ell } + v_{n+M-\ell } + q^{\lambda '+\ell } w_{n+M-\ell } \big) c_{\ell } .
\end{gather*}
If $\lambda ' +n$ is an exponent for some positive integer $n$, then we need the equation
\begin{gather}
 \sum_{\ell =0}^{n-1} \big( q^{-\lambda ' -\ell } u_{n+M-\ell } + v_{n+M-\ell } + q^{\lambda ' +\ell } w_{n+M-\ell } \big) c_{\ell } =0 \label{eq:appsing}
\end{gather}
in order to have series solutions. Otherwise we need logarithmic terms for the solutions.
If the difference of the exponents is a non-zero integer and equation~(\ref{eq:appsing}) is satisfied,
then the singularity $x=0$ is called apparent (or non-logarithmic).

Next, we assume that the point $x=\infty $ of equation~(\ref{eq:axgbxgcxg1}) is a regular singularity, i.e., $N=N'' \geq N'$ in equation~(\ref{eq:uvw}), and investigate local solutions about $x= \infty$ of the form
\begin{gather*}
 h(x)= (1/x)^{\lambda } \sum _{\ell =0}^{\infty} \tilde{c}_{\ell } (1/x)^{\ell }, \qquad \tilde{c}_0 \neq 0.
\end{gather*}
Then we have
\begin{gather*}
 \sum _{\ell =0}^{n} \big( q^{\lambda +\ell } u_{N+\ell -n} + v_{N+\ell -n} + q^{-\lambda -\ell } w_{N+\ell -n} \big) \tilde{c}_{\ell } =0,
\end{gather*}
where $n$ is a non-negative integer. By setting $n=0$, we obtain the characteristic equation at the regular singularity $x=\infty $ as follows
\begin{gather}
 q^{\lambda } u_N + v_N + q^{-\lambda } w_N =0.\label{eq:expoinf}
\end{gather}
The exponents about the regular singularity $x=\infty $ are the values $\lambda $ which satisfy the characteristic equation, i.e., equation~(\ref{eq:expoinf}). Assume that $\lambda '$ is an exponent at $x=\infty $ and $\lambda ' +n$ is also an exponent at $x=\infty $ for some positive integer $n$. Then the singularity is called apparent, iff
\begin{gather*}
 \sum _{\ell =0}^{n-1} \big( q^{\lambda +\ell }u_{N+\ell -n} + v_{N+\ell -n} + q^{-\lambda -\ell } w_{N+\ell -n} \big) \tilde{c}_{\ell } =0.
\end{gather*}

\section[Local behaviors of solutions to the $q$-Heun equation and its variants]{Local behaviors of solutions to the $\boldsymbol{q}$-Heun equation\\ and its variants} \label{eq:locqH}

We investigate local behaviors of solutions about $x=0$ and $x=\infty $ to the $q$-Heun equation and its variants.
We obtain theorems that the local behaviors of solutions determine the coefficients of the $q$-difference equation for the cases of the variants of the $q$-Heun equation.

\subsection{}
Let $A ^{\langle 4 \rangle} $ be the operator in equation~(\ref{eq:qH}). Then the $q$-Heun equation is written as
\begin{gather}
\big( A^{\langle 4 \rangle} -E\big) g(x)=0,\label{eq:A4E}
\end{gather}
where $E$ is a constant. The singularities $x=0$ and $x=\infty $ of the $q$-Heun equation are regular, and the constants~$M$ and~$N$ in Section~\ref{sec:regsing} are given by $M=-1$ and $N=1$. The operator~$A ^{\langle 4 \rangle} $ acts on $x^{\mu } $ as
\begin{gather*}
 A^{\langle 4 \rangle} x^{\mu } = d^{\langle 4 \rangle, +} (\mu ) x^{\mu +1 } + d^{\langle 4 \rangle, 0} (\mu ) x^{\mu } + d^{\langle 4 \rangle, -} (\mu ) x^{\mu -1} , %\label{eq:A4xmu}
\end{gather*}
where
\begin{gather}
 d^{\langle 4 \rangle, +} (\mu ) = q^{\alpha _1 +\alpha _2} q^{\mu } - \big(q^{\alpha _1} +q^{\alpha _2} \big) + q^{-\mu } ,\nonumber \\
 d^{\langle 4 \rangle, 0} (\mu ) = -\big( q^{h_1 +1/2} t_1 +q^{h_2 +1/2} t_2 \big) q^{-\mu } -\big( q^{l_1-1/2} t_1 +q^{l_2-1/2} t_2 \big) q^{\alpha _1 +\alpha _2 + \mu } , \label{eq:d4+} \\
 d^{\langle 4 \rangle, -} (\mu ) =\big\{ q^{h_1 + h_2 + 1} q^{-\mu }\! -q^{( h_1 +h_2 + l_1 + l_2 +\alpha _1 +\alpha _2)/2 } \big( q^{\beta /2}+ q^{-\beta/2}\big)+ q^{l_1 +l_2 +\alpha _1 +\alpha _2 -1 } q^{\mu } \big\} t_1 t_2 . \nonumber
\end{gather}
We investigate the singularity $x=0$. Set
\begin{gather}
 g(x)= x^{\lambda } \sum _{n=0}^{\infty} c_n x^n, \qquad c_0 \neq 0, \label{eq:gxx0}
\end{gather}
and substitute this into equation~(\ref{eq:A4E}).
Then we have
\begin{gather*}
c_0 d^{\langle 4 \rangle, -} (\lambda ) =0, \qquad c_1 d^{\langle 4 \rangle, -} (\lambda + 1 ) + c_0 \big(d^{\langle 4 \rangle, 0} (\lambda ) -E\big) =0, \\
c_n d^{\langle 4 \rangle, -} (\lambda +n ) + c_{n-1} \big(d^{\langle 4 \rangle, 0} (\lambda +n -1 ) -E\big) +c_{n-2} d^{\langle 4 \rangle, +} (\lambda +n -2 ) =0 , \qquad n\geq 2 .
\end{gather*}
The characteristic equation about $x=0$ is written as $d^{\langle 4 \rangle, -} (\lambda )=0 $, i.e.,
\begin{gather*}
q^{h_1 + h_2 + 1} q^{-\lambda } -q^{(h_1 +h_2 + l_1 + l_2 +\alpha _1 +\alpha _2 )/2 } \big( q^{\beta /2}+ q^{-\beta/2}\big)+ q^{l_1 +l_2 +\alpha _1 +\alpha _2 -1 } q^{\lambda } =0.
\end{gather*}
Hence the values
\begin{gather}
\lambda _1 = (h_1 +h_2 -l_1-l_2 -\alpha _1-\alpha _2 -\beta +2)/2, \nonumber\\
\lambda _2 = (h_1 +h_2 -l_1-l_2 -\alpha _1-\alpha _2 +\beta +2)/2 \label{eq:A4la1la2}
\end{gather}
are exponents about $x=0$. If $\lambda _2 - \lambda _1 = \beta $ is not a positive integer, then we have $d^{\langle 4 \rangle, -} (\lambda _1 +n ) \neq 0$ for $n \in \Zint _{>0}$ and the coefficients $c_n $ for $\lambda =\lambda _1 $ and $n\geq 1$ are determined recursively. If $\lambda _1 - \lambda _2 \not \in \Zint $, then we have solutions to equation~(\ref{eq:A4E}) written as
\begin{gather*}
 x^{\lambda _1 } \left\{ 1+ \sum _{n=1}^{\infty} c_n x^{n} \right\} , \qquad x^{\lambda _2 } \left\{ 1 + \sum _{n=1}^{\infty} c'_n x^{n} \right\} .
\end{gather*}

We investigate equation~(\ref{eq:A4E}) about $x=\infty $. Set
\begin{gather}
g(x)= (1/x)^{\lambda } \sum _{n=0}^{\infty} \tilde{c}_n (1/x)^n= \sum _{n=0}^{\infty} \tilde{c}_n x^{-\lambda -n}, \qquad \tilde{c}_0 \neq 0.\label{eq:gxxinfty}
\end{gather}
By substituting this into equation~(\ref{eq:A4E}), we have
\begin{gather*}
 \tilde{c}_0 d^{\langle 4 \rangle, +} (-\lambda ) =0, \qquad \tilde{c}_1 d^{\langle 4 \rangle, +} (-\lambda - 1 ) + \tilde{c}_0 \big(d^{\langle 4 \rangle, 0} (-\lambda ) -E\big) =0, \\
 \tilde{c}_n d^{\langle 4 \rangle, +} (-\lambda -n ) + \tilde{c}_{n-1} \big(d^{\langle 4 \rangle, 0} (-\lambda -n +1 ) -E\big) +\tilde{c}_{n-2} d^{\langle 4 \rangle, +} (-\lambda -n +2) =0 . \nonumber
\end{gather*}
The characteristic equation at $x=\infty $ is written as $d^{\langle 4 \rangle, +} (-\lambda )=0 $, i.e.,
\begin{gather*}
q^{\alpha _1 +\alpha _2} q^{-\lambda } - \big(q^{\alpha _1} +q^{\alpha _2} \big) + q^{\lambda } =0 .
\end{gather*}
Hence the values $ \alpha _1 $ and $ \alpha _2 $ are exponents about $x=\infty $. If $ \alpha _2 -\alpha _1 $ is not a positive integer, then we have $d^{\langle 4 \rangle, +} (-\alpha _1 -n ) \neq 0$ for $n \in \Zint _{>0}$ and the coefficients $\tilde{c}_n $, $n\geq 1$, are determined recursively. If $\alpha _1 -\alpha _2 \not \in \Zint$, then we have solutions to equation~(\ref{eq:A4E}) written as
\begin{gather*}
 x^{-\alpha _1 } \left\{ 1+ \sum _{n=1}^{\infty} \tilde{c}_n x^{-n} \right\} , \qquad x^{-\alpha _2 } \left\{ 1 + \sum _{n=1}^{\infty} \tilde{c}'_n x^{-n} \right\} .
\end{gather*}

Equation~(\ref{eq:A4E}) is written as
\begin{gather*}
a(x) g(x/q) + b(x) g(x) + c(x) g(qx) =0,%\label{eq:axgbxgcxgA4}
\end{gather*}
where
\begin{gather*}
 a(x) = \big(x- q^{h_1+1/2} t_1\big) \big(x- q^{h_2+1/2} t_2\big) , \qquad %\label{eq:acbA4E} \\
 b(x) = b_2 x^2 +E x + b_0 , \nonumber \\
 c(x) = q^{\alpha _1 +\alpha _2} \big(x- q^{l_1 -1/2} t_1 \big)\big(x- q^{l_2 -1/2}t_2 \big), \nonumber
\end{gather*}
and
\begin{gather}
 b_2= -\big(q^{\alpha _1} +q^{\alpha _2} \big), \qquad b_0 = - q^{(h_1 +h_2 +l_1 +l_2 +\alpha _1 + \alpha _2 )/2} \big( q^{\beta/2} + q^{-\beta/2} \big) t_1 t_2 . \label{eq:b20A4E}
\end{gather}
If the exponents of equation~(\ref{eq:A4E}) about $x=0$ (resp.\ $x=\infty $) are $\lambda _1 $ and $\lambda _2 $ in equation~(\ref{eq:A4la1la2}) (resp.\ $\alpha _1$ and $\alpha _2$), then $b_0$ (resp.\ $b_2$) is given by equation~(\ref{eq:b20A4E}). The parameter $E$ is independent of the exponents about $x=0, \infty $, and it is reasonable to regard $E$ as an accessory parameter.

It was shown in \cite{TakR} that the $q$-Heun equation reduces to Heun's differential equation in the $q\to 1$ limit.

\subsection{} \label{sec:A3}
Let $A ^{\langle 3 \rangle} $ be the operator in equation~(\ref{eq:qthird}). Then the equation
\begin{gather}
\big( A^{\langle 3 \rangle} -E\big) g(x)=0\label{eq:A3E}
\end{gather}
for a fixed constant $E$ is a variant of the $q$-Heun equation. The singularities $x=0$ and $x=\infty $ of equation~(\ref{eq:A3E}) are regular, and the constants $M$ and $N$ in Section~\ref{sec:regsing} are given by $M=-1$ and $N=2$. The operator $A ^{\langle 3 \rangle} $ acts on $x^{\mu } $ as
\begin{gather}
A^{\langle 3 \rangle} x^{\mu } = d^{\langle 3 \rangle, ++} (\mu ) x^{\mu +2 } + d^{\langle 3 \rangle, +} (\mu ) x^{\mu +1 } + d^{\langle 3 \rangle, 0} (\mu ) x^{\mu } + d^{\langle 3 \rangle, -} (\mu ) x^{\mu -1} , \label{eq:A3xmu}
\end{gather}
where
\begin{gather}
 d^{\langle 3 \rangle, ++} (\mu ) = q^{-\mu } + q^{\mu } -q^{1/2} -q^{-1/2} , \nonumber \\
 d^{\langle 3 \rangle, +} (\mu ) = \sum _{i=1}^3 \big( q^{h_i} +q^{l_i} - q^{h_i+1/2 -\mu } - q^{l_i-1/2 +\mu } \big) t_i , \nonumber \\
 d^{\langle 3 \rangle, 0} (\mu ) = \sum _{1\leq i<j \leq 3} \big( q^{h_i+h_j+1-\mu } + q^{l_i+l_j-1+\mu } \big) t_i t_j ,\nonumber \\
 d^{\langle 3 \rangle, -} (\mu ) = \big\{ {-}q^{h_1+h_2+h_3 +3/2 -\mu } - q^{l_1+l_2+l_3 -3/2 +\mu } \nonumber\\
 \hphantom{d^{\langle 3 \rangle, -} (\mu ) =}{} + q^{(l_1 +l_2 +l_3 +h_1 +h_2 +h_3)/2} \big( q^{\beta/2} + q^{-\beta/2} \big) \big\} t_1 t_2 t_3 .\label{eq:d3++}
\end{gather}
We investigate the singularity $x=0$ of equation~(\ref{eq:A3E}). We substitute equation~(\ref{eq:gxx0}) into equation~(\ref{eq:A3E}). Then we have
\begin{gather*}
d^{\langle 3 \rangle, -} (\lambda ) =0, \qquad c_1 d^{\langle 3 \rangle, -} (\lambda + 1 ) =- c_0 \big(d^{\langle 3 \rangle, 0} (\lambda ) -E\big) , \\
 c_2 d^{\langle 3 \rangle, -} (\lambda +2 ) =-\big\{ c_{1} \big(d^{\langle 3 \rangle, 0} (\lambda +1 ) -E\big) +c_{0} d^{\langle 3 \rangle, +} (\lambda ) \big\} , \\
c_n d^{\langle 3 \rangle, -} (\lambda +n ) =-\big\{ c_{n-1} \big(d^{\langle 3 \rangle, 0} (\lambda +n -1 ) -E\big) +c_{n-2} d^{\langle 3 \rangle, +} (\lambda +n -2) \\
\hphantom{c_n d^{\langle 3 \rangle, -} (\lambda +n ) =}{} +c_{n-3} d^{\langle 3 \rangle, ++} (\lambda +n -3) \big\} , \qquad n \geq 3.
\end{gather*}
The characteristic equation at $x=0$ is written as $d^{\langle 3 \rangle, -} (\lambda )=0 $, i.e.,
\begin{gather*}
-q^{h_1+h_2+h_3 +3/2 -\lambda } - q^{l_1+l_2+l_3 -3/2 +\lambda } + q^{(l_1 +l_2 +l_3 +h_1 +h_2 +h_3)/2} \big( q^{\beta/2} + q^{-\beta/2} \big) =0 .
\end{gather*}
Hence the values
\begin{gather*}
\lambda _1 = (h_1 +h_2 +h_3 -\beta -l_1-l_2-l_3 +3)/2 , \\
 \lambda _2 = (h_1 +h_2 +h_3 +\beta-l_1-l_2-l_3 +3)/2 %\label{eq:A3la1la2}
\end{gather*}
are exponents about $x=0$.

We investigate the singularity $x=\infty $ of equation~(\ref{eq:A3E}). We substitute equation~(\ref{eq:gxxinfty}) into equation~(\ref{eq:A3E}). Then we have
\begin{gather*}
 d^{\langle 3 \rangle, ++} (-\lambda ) =0, \qquad \tilde{c}_1 d^{\langle 3 \rangle, ++} (-\lambda - 1 ) =- \tilde{c}_0 d^{\langle 3 \rangle, +} (-\lambda ) , \\
\tilde{c}_2 d^{\langle 3 \rangle, ++} (-\lambda -2 ) = -\big\{ \tilde{c}_{1} d^{\langle 3 \rangle, +} (-\lambda -1 ) + \tilde{c}_{0} (d^{\langle 3 \rangle, 0} (-\lambda ) -E) \big\}, \nonumber \\
 \tilde{c}_n d^{\langle 3 \rangle, ++} (-\lambda -n ) =-\big\{ \tilde{c}_{n-1} d^{\langle 3 \rangle, +} (-\lambda -n+1 )+ \tilde{c}_{n-2} \big(d^{\langle 3 \rangle, 0} (-\lambda -n +2 ) -E\big) \nonumber \\
\hphantom{\tilde{c}_n d^{\langle 3 \rangle, ++} (-\lambda -n ) =}{} +\tilde{c}_{n-3} d^{\langle 3 \rangle, +} (-\lambda -n +3)\big\} , \qquad n \geq 3 . \nonumber
\end{gather*}
The singularity $x=\infty $ is regular and the characteristic equation at $x=\infty $ is written as $d^{\langle 3 \rangle, ++} (-\lambda )=0 $, i.e.,
\begin{gather*}
q^{\lambda } + q^{-\lambda } -q^{1/2} -q^{-1/2} =0 .
\end{gather*}
Hence the values $\lambda _1 = -1/2 $ and $\lambda _2 = 1/2 $ are exponents about $x=\infty $, and the difference of the exponents is $1$. In the case $\lambda _1 = -1/2 $, we have $d^{\langle 3 \rangle, ++} (-\lambda _1 - 1) =0$ and it follows from the relation $ d^{\langle 3 \rangle, +} (-\lambda _1 ) =0 $ (see equation~(\ref{eq:d3++})) that the regular singularity $x=\infty $ is apparent. We have two independent solutions written as
\begin{gather*}
x^{1/2} \left\{ 1+ \sum _{n=2}^{\infty} \tilde{c}_n x^{-n} \right\} , \qquad x^{-1/2} \left\{ 1 + \sum _{n=1}^{\infty} {\tilde{c}}'_n x^{-n} \right\} .
\end{gather*}

Recall that equation~(\ref{eq:A3E}) is written as
\begin{gather}
a(x) g(x/q) + b(x) g(x) + c(x) g(qx) =0, \label{eq:axgbxgcxg0}
\end{gather}
where
\begin{gather}
a(x)= \big(x- q^{h_1+1/2} t_1\big) \big(x- q^{h_2+1/2} t_2\big) \big(x- q^{h_3+1/2} t_3\big) , \nonumber \\
b(x) = b_3 x^3 +b_2 x^2 +E x + b_0 , \nonumber \\
c(x)= \big(x- q^{l_1 -1/2} t_1 \big)\big(x- q^{l_2 -1/2}t_2 \big)\big(x- q^{l_3 -1/2}t_3 \big),\label{eq:acbA3E}
\end{gather}
and
\begin{gather}
b_3= -\big(q^{1/2} +q^{-1/2} \big), \qquad b_2= \sum _{n=1}^3 \big( q^{h_n} + q^{l_n} \big) t_n, \nonumber\\
 b_0 = q^{(l_1 +l_2 +l_3 +h_1 +h_2 +h_3)/2} \big( q^{\beta/2} + q^{-\beta/2} \big) t_1 t_2 t_3 . \label{eq:b320A3E}
\end{gather}
Note that the parameter $E$ is independent of the exponents about $x=0, \infty $ and apparency of the regular singularity $x=\infty $. Therefore it is reasonable to regard $E$ as an accessory parameter.

Next we consider a characterization of a variant of the $q$-Heun equation in equation~(\ref{eq:A3E}) by restricting the polynomials $a(x)$, $b(x)$ and $c(x)$ in the difference equation (\ref{eq:axgbxgcxg0}) which satisfy $\deg _x a(x)= \deg _x c(x)=3$, $\deg _x b(x) \leq 3 $ and $a(0 ) \neq 0 \neq c(0)$. If the function $g(x)$ satisfies equation~(\ref{eq:axgbxgcxg0}) and $g(x)= x^{\nu } h(x)$, then the function $h(x)$ satisfies
\begin{gather}
q^{-\nu} a(x) h(x/q) + b(x) h(x) + q^{\nu } c(x) h(qx) =0.\label{eq:axgbxgcxh0}
\end{gather}
Therefore we may assume that the polynomials $a(x)$ and $c(x)$ are monic by a gauge transformation. We denote the zeros of the polynomials $a(x)$ and $c(x)$ by equation~(\ref{eq:acbA3E}). Then the coefficients $b_3$, $b_2$ and $b_0$ are determined by imposing conditions of local behaviors about $x=0$ and $x=\infty $. Namely we have the following theorem.
\begin{Theorem} \label{thm:A3} Assume that the functions $a(x)$, $b(x)$ and $c(x)$ take the form of equation~\eqref{eq:acbA3E}. If the difference of the exponents about $x=0$ is $\beta (\neq 0)$, the difference of the exponents about $x=\infty $ is~$1$, and the regular singularity $x=\infty $ is apparent, then we obtain equation~\eqref{eq:b320A3E} and recover a variant of the $q$-Heun equation in equation~\eqref{eq:A3E}.
\end{Theorem}
\begin{proof} Let $\lambda $ and $\lambda +\beta $ be the exponents about $x=0$. Then we have
\begin{gather*}
- q^{h_1+ h_2+ h_3+3/2 -\lambda } t_1 t_2 t_3 - q^{l_1+ l_2+ l_3 - 3/2 +\lambda } t_1 t_2 t_3 +b_0=0 , \\
- q^{h_1+ h_2+ h_3+3/2 -\lambda -\beta } t_1 t_2 t_3 - q^{l_1+ l_2+ l_3 - 3/2 +\lambda +\beta } t_1 t_2 t_3 +b_0=0.
\end{gather*}
Hence we have $\big(q^{h_1+ h_2+ h_3+3/2 -\lambda } - q^{l_1+ l_2+ l_3 - 3/2 +\lambda +\beta }\big)\big(1-q^{-\beta} \big)=0$, $\lambda = (l_1 +l_2 +l_3 +h_1 +h_2 +h_3 + 3 -\beta )/2 $ and $ b_0 = q^{(l_1 +l_2 +l_3 +h_1 +h_2 +h_3)/2} \big( q^{\beta/2} + q^{-\beta/2} \big) t_1 t_2 t_3 $. Similarly it follows from the condition of the exponents about $x=\infty $ that $b_3= -\big(q^{1/2} +q^{-1/2} \big)$ and the exponents are $1/2$ and $-1/2$. The condition that the singularity $x=\infty $ is apparent is written as $-q^{-1/2}\big(q^{h_1+1/2} t_1 + q^{h_2+1/2}t_2 + q^{h_3+1/2}t_3 \big) + b_2 -q^{1/2}\big(q^{l_1-1/2}t_1 + q^{l_2-1/2}t_2 + q^{l_3-1/2}t_3 \big) =0 $. Therefore we obtain equation~(\ref{eq:b320A3E}).
\end{proof}

We are going to obtain a differential equation from the difference equation~(\ref{eq:A3E}) by taking a~suitable limit. We rewrite equation~(\ref{eq:A3E}) as
\begin{gather}
 \big(x-t_1 q^{h_1 +1/2} \big) \big(x-t_2 q^{h_2 +1/2} \big) \big(x-t_3 q^{h_3 +1/2} \big) g(x/q) \nonumber\\
\qquad{} + \big(x - t_1 q^{l_1 -1/2} \big) \big(x - t_2 q^{l_2 -1/2} \big) \big(x - t_3 q^{l_3 -1/2} \big) g(xq) \nonumber \\
\qquad{} + \bigg[ {-}\big(q^{1/2} +q^{-1/2} \big) x^3 +\sum _{n=1}^3 t_n \big( q^{h_n} + q^{l_n} \big) x^2 - \big\{ 2 (t_1 t_2 +t_2 t_3 +t_3 t_1 ) +(q-1)E_1 \nonumber \\
 \qquad{} +(q-1)^2 \tilde{E} \big\} x + t_1 t_2 t_3 q^{(l_1 +l_2 +l_3 +h_1 +h_2 +h_3)/2} \big( q^{\beta/2} + q^{-\beta/2} \big) \bigg] g(x) =0,\label{eq:qHeunA3}
\end{gather}
where
\begin{gather*}
E_1= (h_1 +h_2 +l_1 +l_2 )t_1 t_2 +(h_2 +h_3 +l_2 +l_3 )t_2 t_3 +(h_3 +h_1 +l_3 +l_1 )t_3 t_1 .
\end{gather*}
Set $q=1+ \varepsilon $. We divide equation~(\ref{eq:qHeunA3}) by $\varepsilon ^2$. By Taylor's expansion
\begin{gather*}
 g(x/q) =g(x) + \big({-}\varepsilon +\varepsilon ^2 \big)xg'(x) + \varepsilon ^2 x^2 g''(x) /2 +O\big( \varepsilon ^3\big),\\
 g(qx) =g(x) + \varepsilon x g'(x) + \varepsilon ^2 x^2 g''(x) /2 +O\big( \varepsilon ^3\big),
\end{gather*}
we find the following limit as $\varepsilon \to 0$:
\begin{gather}
 x^2(x-t_1)(x-t_2) (x-t_3 ) g''(x) \nonumber\\
\qquad{} + x^2(x-t_1)(x-t_2) (x-t_3 ) \left[ \frac{h_1 - l_1 +1 }{x- t_1} + \frac{h_2 - l_2 +1 }{x- t_2} +\frac{h_3 - l_3 +1 }{x- t_3} - \frac{2\tilde{l}}{x} \right] g'(x) \nonumber \\
\qquad{} - \big[ x^3 /4 + \{ (2 h_1 -2 l_1 +1)t_1+(2 h_2 -2 l_2 +1)t_2 +(2 h_3 -2 l_3 +1) t_3 \} x^2/4 \nonumber \\
\qquad{} + \tilde{B} x + t_1 t_2 t_3\big(\tilde{l} +1/2 +\beta /2\big) \big(\tilde{l} + 1/2 -\beta /2\big) \big]g(x) =0, \label{eq:Fuchs4sing}
\end{gather}
where
\begin{gather*}
 \tilde{l} = (h_1 + h_2 +h_3 -l_1 -l_2 -l_3 +2)/2 ,\qquad \tilde{B} = \tilde{E} + \mbox{const} .
\end{gather*}
This is a Fuchsian differential equation with five singularities $\{ 0,t_1, t_2, t_3, \infty \}$ and the local exponents are given by the following Riemann scheme
\begin{gather*}
\begin{pmatrix}
x=0 & x=t_1 & x=t_2 & x=t_3 & x=\infty \\
\tilde{l} +\beta /2 +1/2 & 0 & 0 & 0 & - 1/2 \\
\tilde{l} - \beta /2 +1/2 & l_1-h_1 & l_2-h_2 & l_3-h_3 & 1/2
\end{pmatrix} .
\end{gather*}
By the linear fractional transformation $z=t_2 (1-t_1/x)/(t_2-t_1)$ and the gauge transformation $g(x)=x^{1/2 }\tilde{g}(z)$, the function $y=\tilde{g}(z) $ satisfies Heun's differential equation
\begin{gather*}
\frac{{\rm d}^2y}{{\rm d}z^2} + \left( \frac{\gamma}{z}+\frac{\delta }{z-1}+\frac{\epsilon}{z-t}\right) \frac{{\rm d}y}{{\rm d}z} + \frac{\alpha ' \beta ' z -q}{z(z - 1)(z - t)} y= 0,%\label{eq:Heun1}
\end{gather*}
where $t=t_2(t_1-t_3)/\{ t_3(t_1-t_2) \}$, $\gamma = 1 +h_1 -l_1$, $\delta = 1 +h_2 -l_2$, $\epsilon = 1+h_3-l_3$, $\{ \alpha ' ,\beta ' \} = \{ \tilde{l} -\beta /2, \tilde{l} +\beta /2 \}$, $q= \tilde{E}/(t_3(t_2 -t_1)) + \mbox{const}$. Hence the point $x=\infty $ in equation~(\ref{eq:Fuchs4sing}) is essentially non-singular (apparent), and we may say that the equation $( A^{\langle 3 \rangle} -E) g(x)=0$ is a~$q$-analogue of the Fuchsian equation with four regular singularities $\{ 0, t_1, t_2 ,t_3 \}$, $t_1, t_2, t_3 \in \C _{\neq 0}$.

\subsection{} \label{sec:A2}
Let $A ^{\langle 2 \rangle} $ be the operator in equation~(\ref{eq:qsecond}). Then the equation
\begin{gather}
\big( A^{\langle 2 \rangle} -E\big) g(x)=0.\label{eq:A2E}
\end{gather}
for a fixed constant $E$ is a variant of the $q$-Heun equation. The singularities $x=0$ and $x=\infty $ of equation~(\ref{eq:A2E}) are regular, and the constants $M$ and $N$ in Section~\ref{sec:regsing} are given by $M=-2$ and $N=2$. The operator $A ^{\langle 2 \rangle} $ acts on $x^{\mu } $ as
\begin{gather}
 A^{\langle 2 \rangle} x^{\mu } = d^{\langle 2 \rangle, ++} (\mu ) x^{\mu +2 } + d^{\langle 2 \rangle, +} (\mu ) x^{\mu +1 } + d^{\langle 2 \rangle, 0} (\mu ) x^{\mu }\nonumber\\
 \hphantom{A^{\langle 2 \rangle} x^{\mu } =}{} + d^{\langle 2 \rangle, -} (\mu ) x^{\mu -1} + d^{\langle 2 \rangle, --} (\mu ) x^{\mu -2} , \label{eq:A2xmu}
\end{gather}
where
\begin{gather*}
 d^{\langle 2 \rangle, ++} (\mu ) = q^{-\mu } + q^{\mu } -q^{1/2} -q^{-1/2} ,\\ % \label{eq:d2++} \\
 d^{\langle 2 \rangle, +} (\mu ) = \sum _{i=1}^4 \big( q^{h_i} +q^{l_i} - q^{h_i+1/2 -\mu } - q^{l_i-1/2 +\mu } \big) t_i , \nonumber \\
 d^{\langle 2 \rangle, 0} (\mu ) = \sum _{1\leq i <j \leq 4} \big( q^{h_i + h_j +1-\mu } + q^{l_i + l_j -1+\mu } \big) t_i t_j , \nonumber \\
 d^{\langle 2 \rangle, -} (\mu ) = t_1 t_2 t_3 t_4 \sum _{i=1}^4 \big\{ q^{(h_1 +h_2 +h_3 +h_4 +l_1 +l_2 +l_3 +l_4)/2} \big( q^{-h_i}+ q^{-l_i}\big) \nonumber \\
 \hphantom{d^{\langle 2 \rangle, -} (\mu ) =}{} - q^{h_1 + h_2 +h_3 + h_4 + 3/2 -\mu } q^{-h_i } - q^{l_1 + l_2 +l_3 + l_4 - 3/2 +\mu } q^{-l_i } \big\} t_i^{-1} , \nonumber \\
 d^{\langle 2 \rangle, --} (\mu ) = \big\{ q^{h_1+h_2+h_3 +h_4+2 -\mu }+q^{l_1+l_2+l_3 +l_4-2 +\mu } \nonumber \\
 \hphantom{d^{\langle 2 \rangle, --} (\mu ) =}{} - q^{(l_1 +l_2 +l_3 +l_4+h_1 +h_2 +h_3+l_4)/2} \big( q^{1/2} + q^{-1/2} \big) \big\} t_1t_2t_3t_4 . \nonumber
\end{gather*}
We investigate the singularity $x=0$ of equation~(\ref{eq:A2E}). We substitute equation~(\ref{eq:gxx0}) into equation~(\ref{eq:A2E}). Then we have
\begin{gather*}
 d^{\langle 2 \rangle, --} (\lambda ) =0, \qquad c_1 d^{\langle 2 \rangle, --} (\lambda + 1 ) =- c_0 d^{\langle 2 \rangle, -} (\lambda ) , \\
 c_n d^{\langle 3 \rangle, --} (\lambda +n ) =-\big\{ c_{n-1} d^{\langle 3 \rangle, -} (\lambda +n -1 ) +c_{n-2} \big(d^{\langle 3 \rangle, 0} (\lambda +n -2 ) -E\big) \nonumber \\
\hphantom{c_n d^{\langle 3 \rangle, --} (\lambda +n ) =}{} +c_{n-3} d^{\langle 3 \rangle, +} (\lambda +n -3) +c_{n-4} d^{\langle 3 \rangle, ++} (\lambda +n -4) \big\} , \qquad n \geq 2 , \nonumber
\end{gather*}
where we set $c_{-1}=c_{-2}=0 $. The singularity $x=0$ is regular and the characteristic equation at $x=0$ is written as $d^{\langle 2 \rangle, --} (\lambda )=0 $, i.e.,
\begin{gather*}
q^{h_1+h_2+h_3 +h_4+2 -\lambda}+q^{l_1+l_2+l_3 +l_4-2 +\lambda}- q^{(l_1 +l_2 +l_3 +l_4+h_1 +h_2 +h_3+l_4)/2} \big( q^{1/2} + q^{-1/2} \big) =0.
\end{gather*}
Hence the values
\begin{gather*}
\lambda _1 = (h_1 +h_2 +h_3 +h_4 -l_1-l_2-l_3 -l_4 +3)/2 , \qquad \lambda _2 = \lambda _1 +1%\label{eq:A2la1la2}
\end{gather*}
are exponents about $x=0$, i.e., $ d^{\langle 2 \rangle, --} (\lambda _1 ) = d^{\langle 2 \rangle, --} (\lambda _2 )=0$. It follows from $d^{\langle 2 \rangle, -} (\lambda _1 ) =0$ that the singularity $x=0$ is apparent, and we have two independent solutions written as
\begin{gather*}
 x^{\lambda _1 } \left\{ 1+ \sum _{n=2}^{\infty} c_n x^{n} \right\} , \qquad x^{\lambda _1 +1} \left\{ 1 + \sum _{n=1}^{\infty} c'_n x^{n} \right\} .
\end{gather*}

We investigate the singularity $x=\infty $ of equation~(\ref{eq:A2E}). We substitute equation~(\ref{eq:gxxinfty}) into equation~(\ref{eq:A2E}). Then we have
\begin{gather*}
 d^{\langle 2 \rangle, ++} (-\lambda ) =0, \qquad \tilde{c}_1 d^{\langle 2 \rangle, ++} (-\lambda - 1 ) =- \tilde{c}_0 d^{\langle 2 \rangle, +} (-\lambda ) , \\
 \tilde{c}_n d^{\langle 2 \rangle, ++} (-\lambda -n ) =-\big\{ \tilde{c}_{n-1} d^{\langle 2 \rangle, +} (-\lambda -n+1 ) + \tilde{c}_{n-2} \big(d^{\langle 2 \rangle, 0} (-\lambda -n +2 ) -E\big) \nonumber \\
\hphantom{\tilde{c}_n d^{\langle 2 \rangle, ++} (-\lambda -n ) =}{} +\tilde{c}_{n-3} d^{\langle 2 \rangle, +} (-\lambda -n +3) +\tilde{c}_{n-4} d^{\langle 2 \rangle, ++} (-\lambda -n +4) \big\}, \qquad n \geq 2 , \nonumber
\end{gather*}
where we set $\tilde{c}_{-1}=\tilde{c}_{-2}=0$. The singularity $x=\infty $ is regular and the characteristic equation at $x=\infty $ is written as $d^{\langle 2 \rangle, ++} (-\lambda )=0 $, i.e.,
\begin{gather*}
q^{\lambda } + q^{-\lambda } -q^{1/2} -q^{-1/2} =0 .
\end{gather*}
Hence the values $\lambda _1 = -1/2 $ and $\lambda _2 = 1/2 $ are exponents about $x=\infty $. It is also shown that the singularity $x=\infty $ is apparent.

We have observed that the regular singularities $x=0$ and $x=\infty $ of equation~(\ref{eq:A2E}) are apparent and the differences of the exponents are $1$. Recall that equation~(\ref{eq:A2E}) is written as
\begin{gather*}
a(x) g(x/q) + b(x) g(x) + c(x) g(qx) =0,%\label{eq:axgbxgcxgA2}
\end{gather*}
where
\begin{gather}
a(x)= \big(x- q^{h_1+1/2} t_1\big) \big(x- q^{h_2+1/2} t_2\big) \big(x- q^{h_3+1/2} t_3\big) \big(x- q^{h_4+1/2} t_4\big) , \nonumber\\
b(x) = b_4 x^4 +b_3 x^3 +E x^2+ b_1 x + b_0 , \nonumber \\
c(x)= \big(x- q^{l_1 -1/2}t_1\big)\big(x- q^{l_2 -1/2}t_2\big)\big(x- q^{l_3 -1/2}t_3\big)\big(x- q^{l_4 -1/2}t_4\big),\label{eq:acbA2E}
\end{gather}
and
\begin{gather}
b_4= -\big(q^{1/2} +q^{-1/2} \big), \qquad b_3= \sum _{n=1}^4 \big( q^{h_n} + q^{l_n} \big) t_n, \nonumber\\
b_1 = q^{(l_1 +l_2 +l_3 +l_4 +h_1 +h_2 +h_3+h_4)/2} t_1 t_2 t_3 t_4 \sum _{n=1}^4 \big( q^{-h_n} + q^{-l_n} \big) t_n^{-1}, \nonumber \\
b_0 = - q^{(l_1 +l_2 +l_3 +l_4 +h_1 +h_2 +h_3+h_4)/2} \big( q^{1/2} + q^{-1/2} \big)t_1 t_2 t_3 t_4 .\label{eq:b4310A2E}
\end{gather}
Note that the parameter $E$ is independent of the exponents about $x=0, \infty $ and apparency of the regular singularities $x=0, \infty $. Therefore it is reasonable to regard $E$ as an accessory parameter.

Next we consider a characterization of the coefficients $b_4$, $b_3$, $b_1$ and $b_0$ as an analogue of Theorem~\ref{thm:A3}. Namely we have the following theorem, which is proved similarly to Theorem \ref{thm:A3}.
\begin{Theorem} \label{thm:A2}
Assume that the functions $a(x)$, $b(x)$ and $c(x)$ take the form of equation~\eqref{eq:acbA2E}. If the difference of the exponents about $x=0$ is $1$, the difference of the exponents about $x=\infty $ is $1$, and the regular singularities $x=0, \infty $ are apparent, then we obtain equation~\eqref{eq:b4310A2E} and recover a variant of the $q$-Heun equation in equation~\eqref{eq:A2E}.
\end{Theorem}

We are going to obtain a differential equation from the difference equation~(\ref{eq:A2E}) by taking a suitable limit. We rewrite equation~(\ref{eq:A2E}) as
\begin{gather*}
 \big(x-t_1 q^{h_1 +1/2} \big) \big(x-t_2 q^{h_2 +1/2} \big) \big(x-t_3 q^{h_3 +1/2} \big) \big(x-t_4 q^{h_4 +1/2} \big) g(x/q)\\ %\label{eq:qHeunA2} \\
\qquad{} + \big(x - t_1 q^{l_1 -1/2} \big) \big(x - t_2 q^{l_2 -1/2} \big) \big(x - t_3 q^{l_3 -1/2} \big) \big(x - t_4 q^{l_4 -1/2} \big) g(xq) \nonumber \\
\qquad{} + \bigg[ {}-\big(q^{1/2} +q^{-1/2} \big) x^4 +\sum _{n=1}^4 t_n \big( q^{h_n} + q^{l_n} \big) x^3 \nonumber \\
\qquad{} - \big(2 (t_1 t_2 +t_1 t_3 +t_1 t_4 +t_2 t_3 +t_2 t_4 +t_3 t_4 ) +(q-1)E_1 +(q-1)^2 \tilde{E} \big) x^2 \nonumber \\
\qquad {} + t_1 t_2 t_3 t_4 q^{(l_1 +l_2 +l_3 +l_4 +h_1 +h_2 +h_3 +h_4)/2} \bigg\{ \sum _{n=1}^4 \big( 1/\big(t _n q^{h_n}\big) + 1/\big(t_n q^{l_n}\big) \big) x \nonumber \\
\qquad{} - \big(q^{1/2} +q^{-1/2} \big) \bigg\} \bigg] g(x) =0, \nonumber
\end{gather*}
where
\begin{gather*}
E_1= \sum _{1\leq m<n\leq 4}(h_m +h_n +l_m +l_n )t_m t_n.
\end{gather*}
Set $q=1+ \varepsilon $. Then we obtain the following differential equation after taking the limit $\varepsilon \to 0 $.
\begin{gather}
x^2(x-t_1)(x-t_2) (x-t_3)(x-t_4) \left[ g''(x) + \left\{ - \frac{2\tilde{l}}{x} +\sum _{i=1}^4 \frac{h_i - l_i +1 }{x- t_i} \right\} g'(x) \right]\nonumber \\
\qquad{} - \left[ \frac{x^4}{4} + \left\{ \sum _{i=1}^4 (2 h_i -2 l_i +1)t_i \right\} \frac{x^3}{4} + \tilde{B} x^2 \right.\nonumber \\
\left.\qquad{} + t_1 t_2 t_3 t_4 \tilde{l} \left\{\sum _{i=1}^4 \frac{\tilde{l} -h_i+ l_i }{t_i} \right\} x + t_1 t_2 t_3 t_4 \tilde{l} \big(\tilde{l} +1\big) \right] g(x) =0, \label{eq:Fuchs5sing}
\end{gather}
where
\begin{gather*}
\tilde{l} = (h_1 + h_2 +h_3 +h_4 -l_1 -l_2 -l_3 -l_4 +3)/2 , \qquad \tilde{B} = \tilde{E} +\mbox{const} .
\end{gather*}
This is a Fuchsian differential equation with six singularities $\{ 0,t_1, t_2, t_3, t_4, \infty \}$ and the local exponents are given by the following Riemann scheme
\begin{gather*}
\begin{pmatrix}
x=0 & x=t_1 & x=t_2 & x=t_3 & x=t_4 & x=\infty \\
\tilde{l} & 0 & 0 & 0 & 0 & - 1/2 \\
\tilde{l} +1 & l_1-h_1 & l_2-h_2 & l_3-h_3 & l_4-h_4 & 1/2
\end{pmatrix} .
\end{gather*}
By the linear fractional transformation $z=(x-t_2)(t_3-t_1)/\{ (x-t_1)(t_3-t_2) \} $ and the gauge transformation $g(x)=x^{\tilde{l}} (x-t_1) ^{-\tilde{l} +1/2} \tilde{g}(z)$, the function $y=\tilde{g}(z) $ satisfies Heun's differential equation
\begin{gather*}
\frac{{\rm d}^2y}{{\rm d}z^2} + \left( \frac{\gamma}{z}+\frac{\delta }{z-1}+\frac{\epsilon}{z-t}\right) \frac{{\rm d}y}{{\rm d}z} + \frac{\alpha ' \beta ' z -q}{z(z - 1)(z - t)} y= 0,%\label{eq:Heun2}
\end{gather*}
where $t=(t_4-t_2)(t_3-t_1)/\{ (t_4-t_1)(t_3-t_2) \}$, $\gamma = 1 +h_2 -l_2$, $\delta = 1 +h_3 -l_3$, $\epsilon = 1+h_4-l_4$, $\{ \alpha ' ,\beta ' \} = \{ \tilde{l}-1 /2, \tilde{l} -1/2 + l_1-h_1\}$, $q= \tilde{E}/((t_1-t_4)(t_3 -t_2)) + \mbox{const}$. Hence the points $x=0 , \infty $ in equation~(\ref{eq:Fuchs5sing}) are essentially non-singular (apparent), and we may say that the equation $\big( A^{\langle 2 \rangle} -E\big) g(x)=0$ is a $q$-analogue of the Fuchsian equation with four regular singularities $\{ t_1, t_2 ,t_3 , t_4 \}$, $t_1, t_2, t_3 , t_4 \in \C _{\neq 0}$.

\section[Invariant subspaces related with the $q$-Heun equation and its variants]{Invariant subspaces related with the $\boldsymbol{q}$-Heun equation\\ and its variants} \label{sec:inv}

In order to find special solutions to the $q$-Heun equation or its variants, we consider quasi-exact solvability~\cite{Tur}, i.e., we investigate subspaces which are invariant under the action of the opera\-tors~$A^{\langle 4 \rangle}$, $A^{\langle 3 \rangle}$ or $A^{\langle 2 \rangle}$.

We look for subspaces which are invariant under the action of the operator $A^{\langle 4 \rangle}$ given in equation~(\ref{eq:qH}).
\begin{Proposition}Set
\begin{gather}
\lambda _1 = (h_1 +h_2 -l_1-l_2- \alpha _1- \alpha _2 -\beta +2 )/2 , \nonumber\\
 \lambda _2 = (h_1 +h_2 -l_1-l_2- \alpha _1- \alpha _2 +\beta +2)/2.\label{eq:A4la1la20}
\end{gather}
Let $\lambda \in \{ \lambda _1 ,\lambda _2\}$, $\alpha \in \{ \alpha _1 , \alpha _2 \}$ and assume that $n:=-\lambda - \alpha $ is a non-negative integer. Let $V^{\langle 4 \rangle}$ be the space spanned by the monomials $x^{\lambda +k} $, $k=0,\dots ,n$, i.e.,
\begin{gather*}
V^{\langle 4 \rangle} = \big\{ c_0 x^{\lambda } +c_1 x^{\lambda +1} +\dots + c_n x^{\lambda +n} \,|\, c_0 , c_1 , \dots , c_n \in \C \big\} .
\end{gather*}
Then the operator $A^{\langle 4 \rangle} $ preserves the space $V^{\langle 4 \rangle} $.
\end{Proposition}
\begin{proof}The operator $A ^{\langle 4 \rangle} $ acts on $x^{\lambda +k} $ as
\begin{gather*}
 A^{\langle 4 \rangle} x^{\lambda +k} = d^{\langle 4 \rangle, +} (\lambda +k) x^{\lambda +k +1 } + d^{\langle 4 \rangle, 0} (\lambda +k) x^{\lambda +k } + d^{\langle 4 \rangle, -} (\lambda +k ) x^{\lambda +k -1}.
\end{gather*}
See equation~(\ref{eq:d4+}) for the values $ d^{\langle 4 \rangle, +} (\lambda +k)$, $d^{\langle 4 \rangle, 0} (\lambda +k)$ and $d^{\langle 4 \rangle, -} (\lambda +k)$. Therefore we have $ A^{\langle 4 \rangle} x^{\lambda +k} \in V^{\langle 4 \rangle}$ for $k=1, \dots , n-1$. When $k=0$, we have $d^{\langle 4 \rangle, -} (\lambda ) =0 $ and $A^{\langle 4 \rangle} x^{\lambda } \in V^{\langle 4 \rangle} $. When $k=n$, we have $d^{\langle 4 \rangle, +} (\lambda +n ) = d^{\langle 4 \rangle, +} (-\alpha ) = 0 $ and $A^{\langle 4 \rangle} x^{\lambda + n} \in V^{\langle 4 \rangle} $.
\end{proof}

Note that the values $\lambda _1 $ and $\lambda _2 $ in equation~(\ref{eq:A4la1la20}) are exponents of the $q$-Heun equation (equation~(\ref{eq:A4E})) about $x=0$, and the values $ \alpha _1$ and $\alpha _2 $ are exponents about $x=\infty $.

We look for subspaces which are invariant under the action of the operator $A^{\langle 3 \rangle}$ in equation~(\ref{eq:qthird}).
\begin{Proposition}Set
\begin{gather}
\lambda _1 = (h_1 +h_2 +h_3 -l_1-l_2-l_3 -\beta +3)/2 , \nonumber\\
 \lambda _2 = (h_1 +h_2 +h_3 -l_1-l_2-l_3 +\beta +3)/2 .\label{eq:A3la1la20}
\end{gather}
Let $\lambda \in \{ \lambda _1 ,\lambda _2\}$ and assume that $n:= -\lambda +1/2 $ is a non-negative integer.
Let $V^{\langle 3 \rangle}$ be the space spanned by the monomials $x^{\lambda +k} $, $k=0,\dots ,n$, i.e.,
\begin{gather*}
V^{\langle 3 \rangle} = \big\{ c_0 x^{\lambda } +c_1 x^{\lambda +1} +\dots + c_n x^{1/2} \,|\, c_0 , c_1 , \dots ,c_n \in \C \big\} .
\end{gather*}
Then the operator $A^{\langle 3 \rangle} $ preserves the space $V^{\langle 3 \rangle} $.
\end{Proposition}
\begin{proof}It follows from equation~(\ref{eq:A3xmu}) that
\begin{gather*}
 A^{\langle 3 \rangle} x^{\lambda +k} = d^{\langle 3 \rangle, ++} (\lambda +k) x^{\lambda +k +2 } + d^{\langle 3 \rangle, +} (\lambda +k) x^{\lambda +k +1 } \\
\hphantom{A^{\langle 3 \rangle} x^{\lambda +k} =}{} + d^{\langle 3 \rangle, 0} (\lambda +k) x^{\lambda +k } + d^{\langle 3 \rangle, -} (\lambda +k ) x^{\lambda +k -1} .
\end{gather*}
Since $d^{\langle 3 \rangle, -} (\lambda ) =d^{\langle 3 \rangle, ++} (-1/2 ) = d^{\langle 3 \rangle, ++} (1/2 ) = d^{\langle 3 \rangle, +} (1/2 )= 0 $, we obtain the proposition.
\end{proof}

Note that the values $\lambda _1 $ and $\lambda _2 $ in equation~(\ref{eq:A3la1la20}) are exponents of equation~(\ref{eq:A3E}) about $x=0$, and the values $ -1/2$ and $1/2 $ are exponents about $x=\infty $.

Invariant subspaces of the operator $A^{\langle 2 \rangle}$ in equation~(\ref{eq:qsecond}) are described as follows:
\begin{Proposition} Set
\begin{gather*}
\lambda = (h_1 +h_2 +h_3 + h_4 -l_1-l_2-l_3 -l_4 +3)/2
\end{gather*}
and assume that $n:= -\lambda +1/2 $ is a non-negative integer. Let $V^{\langle 2 \rangle}$ be the space spanned by the monomials $x^{\lambda +k} $, $k=0,\dots ,n$, i.e.,
\begin{gather*}
V^{\langle 2 \rangle} = \big\{ c_0 x^{\lambda } +c_1 x^{\lambda +1} +\dots + c_n x^{1/2} \,|\, c_0 , c_1 , \dots , c_n \in \C \big\} .
\end{gather*}
Then the operator $A^{\langle 2 \rangle} $ preserves the space $V^{\langle 2 \rangle} $.
\end{Proposition}
\begin{proof}
It follows from equation~(\ref{eq:A2xmu}) that
\begin{gather*}
 A^{\langle 2 \rangle} x^{\lambda +k} = d^{\langle 2 \rangle, ++} (\lambda +k) x^{\lambda +k +2 } + d^{\langle 2 \rangle, +} (\lambda +k) x^{\lambda +k +1 } \\
\hphantom{A^{\langle 2 \rangle} x^{\lambda +k} =}{} + d^{\langle 2 \rangle, 0} (\lambda +k) x^{\lambda +k } + d^{\langle 2 \rangle, -} (\lambda +k ) x^{\lambda +k -1} + d^{\langle 2 \rangle, --} (\lambda +k ) x^{\lambda +k -2} .
\end{gather*}
Since $d^{\langle 2 \rangle, --} (\lambda ) =d^{\langle 2 \rangle, -} (\lambda ) = d^{\langle 2 \rangle, --} (\lambda +1 ) = d^{\langle 2 \rangle, ++} (-1/2 ) = d^{\langle 2 \rangle, ++} (1/2 ) = d^{\langle 2 \rangle, +} (1/2 )= 0 $, we obtain the proposition.
\end{proof}

If the dimension of the invariant subspace is one, then the eigenvalue and the eigenfunction of the subspace are calculated explicitly, and they are described as follows:
\begin{Proposition}
Let $\lambda _1$, $\lambda _2 $ be the values in equation~\eqref{eq:A4la1la20}, $\lambda \in \{ \lambda _1 ,\lambda _2\}$, $\alpha \in \{ \alpha _1 , \alpha _2 \}$ and assume that $\lambda =- \alpha $. Then the operator $A^{\langle 4 \rangle} $ preserves the one-dimensional space spanned by the function $x^{\lambda } $, and we have
\begin{gather*}
 A^{\langle 4 \rangle} x^{\lambda } = d^{\langle 4 \rangle, 0} (-\alpha ) x^{\lambda }, \\
 d^{\langle 4 \rangle, 0} (- \alpha ) = -\big( q^{h_1 +1/2} t_1 +q^{h_2 +1/2} t_2 \big) q^{ \alpha } -\big( q^{l_1-1/2} t_1 +q^{l_2-1/2} t_2 \big) q^{\alpha _1 +\alpha _2 - \alpha } .
\end{gather*}
\end{Proposition}
If $\lambda =\lambda _1$, $\alpha =\alpha _1$ and $-\lambda _1 - \alpha _1 =0$, then we recover the eigenfunction of $A^{\langle 4 \rangle} $ in the introduction (see equation~(\ref{eq:A4xla1})).
\begin{Proposition} Assume that $\beta = h_1 +h_2 +h_3 -l_1-l_2-l_3 +2$ or $\beta = -( h_1 +h_2 +h_3 -l_1-l_2-l_3 +2)$. Then the operator $A^{\langle 3 \rangle} $ preserves the one-dimensional space spanned by the function~$x^{1/2 }$, and we have
\begin{gather*}
 A^{\langle 3 \rangle} x^{1/2 } = \bigg[ \sum _{1\leq i <j \leq 3} \big( q^{h_i + h_j +1/2 } + q^{l_i + l_j -1/2 } \big) t_i t_j \bigg] x^{1/2 }.
\end{gather*}
\end{Proposition}
\begin{Proposition} Assume that $h_1 +h_2 +h_3 + h_4 -l_1-l_2-l_3 -l_4 +2=0 $. Then the opera\-tor~$A^{\langle 2 \rangle} $ preserves the one-dimensional space spanned by the function $x^{1/2 } $, and we have
\begin{gather*}
 A^{\langle 2 \rangle} x^{1/2 } = \bigg[ \sum _{1\leq i <j \leq 4} \big( q^{h_i + h_j +1/2 } + q^{l_i + l_j -1/2 } \big) t_i t_j \bigg] x^{1/2 }.
\end{gather*}
\end{Proposition}

\subsection*{Acknowledgements}
The author is grateful to Simon Ruijsenaars for valuable comments and fruitful discussions. He thanks to the referees and the editor for valuable comments. He is supported by JSPS KAKENHI Grant Number JP26400122.

\pdfbookmark[1]{References}{ref}
\LastPageEnding

\end{document}